\documentclass[draft]{publmathdeb} 
\usepackage{amsmath,amsfonts,amssymb}

\usepackage{url}

\def\congruent{\equiv}

\def\ratQ{ \mathbb{Q} }  
 
\def\notdiv{\nmid}

\numberwithin{equation}{section}

\newtheorem{theorem}{Theorem}
\newtheorem{lemma}{Lemma}

\newtheorem*{TheoremA}{Theorem A}
\newtheorem*{Conjecture}{Conjecture}

\def\German{\mathfrak}
\def\lcm{{\rm lcm}}
\def\conj{\overline}
\def\notmid{\nmid}

\begin{document}
	
\title{ Number of solutions to $a^x + b^y = c^z$, A Shorter Version  }

\author{Reese Scott}
\address{86 Boston St., Somerville, MA 02143}

\author{Robert Styer}

\address{Department of Mathematics and Statistics\\
Villanova University \\
800 Lancaster Avenue\\
Villanova, PA  19085--1699 \\
U.S. A.
}

\email{robert.styer@villanova.edu}

{\urladdr{http://www.homepage.villanova.edu/robert.styer/ReeseScott/indexReese.htm} }


\submitted{ February 5, 2015, rev 5 July 2023} 




\begin{abstract} 
For relatively prime integers $a$ and $b$ both greater than one and odd integer $c$, there are at most two solutions in positive integers $(x,y,z)$ to the equation $a^x + b^y = c^z$.  There are an infinite number of $(a,b,c)$ giving exactly two solutions. 
\end{abstract}

\maketitle

\section{Introduction}
(This is a streamlined shortened version of the published paper: Number of Solutions to $a^x + b^y = c^z$, {\it Publ. Math. Debrecen}, Vol 88 (2016), pp. 131-138.)

This paper deals with the problem of finding an upper bound on the number of solutions in positive integers $x$, $y$, and $z$ to the equation
$$a^x + b^y = c^z  \eqno{(1.1)}$$
for integers $a$, $b$, and $c$, all greater than 1 with $\gcd(a,b)=1$.  Although there is much previous work on this problem, realistically low bounds on the number of solutions to (1.1) have been obtained only for the special case in which one of $x$, $y$, or $z$ is constant; results for this special case are obtained using lower bounds on linear forms in logarithms (e.g., \cite{Le1}, \cite{Be}).  For the more general case in which all of $x$, $y$, and $z$ are variable, Mahler \cite{M} used his $p$-adic analogue of the method of Thue-Siegel to prove that (1.1) has only finitely many solutions $(x,y,z)$ (see \cite{Le2}).  Later, Gelfond \cite{Ge} made Mahler's result effective. As pointed out by the anonymous referee of this paper, the existence of a bound on the number of solutions, independent of $a$, $b$, and $c$, follows from a result of Beukers and Schlickewei \cite{BS}.  Hirata-Kohno \cite{HK} used  \cite{BS} to obtain a bound of $2^{36}$ (the referee believes Hirata-Kohno may have later announced a bound of 200, apparently unpublished).  

Le \cite{Le2} dealt with the general case when all of $x$, $y$, and $z$ are variables with $c$ odd:  

\begin{TheoremA}[\cite{Le2}]
	If $2 \notdiv c$ then (1.1) has at most $2^{\omega(c)+1}$ solutions $(x,y,z)$ where $\omega(c)$ is the number of distinct prime factors of $c$. 
	Moreover, all solutions $(x,y,z)$ of (1.1) satisfy $z < (2ab \log(2 e ab))/\pi$.
\end{TheoremA}

We give a brief outline of a proof of Theorem A, which differs somewhat from Le's proof, but will allow us to establish some notation for our own variant of Theorem A. We say that two solutions $(x_1, y_1)$ and $(x_2, y_2)$ to (1.1) are in the same {\it parity class} if $2 \mid x_1 - x_2$ and $2 \mid y_1 - y_2$.    

For each parity class we define $D$ to be the least integer such that $\frac{a^x b^y}{D}$ is a square for any choice of $x$ and $y$ in the parity class.  If $(x,y,z)$ is a solution to (1.1) with $x$ and $y$ in a given parity class, we say that the solution $(x,y,z)$ is in that parity class.  We then define the integer $\gamma(x,y,z)$ in $\ratQ(\sqrt{-D})$:
$$ \gamma(x,y,z) = a^x - b^y + 2 \sqrt{ -a^{x} b^{y} }. \eqno{(1.2)}$$
The norm of $\gamma(x,y,z)$ is $c^{2z}$.  Let $C_D = \{ \German{c}_1 \conj{\German{c}_1},  \German{c}_2 \conj{\German{c}_2}, \dots, \German{c}_g \conj{\German{c}_g} \}$ be the set of all factorizations of $[c]$ into two ideals in $\ratQ(\sqrt{-D})$ such that, for $1 \le i \le g$, no $\German{c}_i$ is divisible by a principal ideal with a rational integer generator.  $g= 2^{\omega(c)-1}$.  
For any $(x,y,z)$,  $[\gamma(x,y,z)] $ must be divisible by exactly one of the ideals $\German{c}_1$, $\conj{\German{c}_1}$,  $\German{c}_2$, $\conj{\German{c}_2}$, \dots, $\German{c}_g$, $\conj{\German{c}_g}$.  We say that the solution $(x,y,z)$ to (1.1) is {\it associated} with the ideal factorization $\German{c}_k \conj{\German{c}_k} $ when either $\German{c}_k $ or $\conj{\German{c}_k} $ divides $[\gamma(x,y,z)]$, where $\German{c}_k \conj{\German{c}_k} \in C_D$.  It is an old result (see Lemma 2 of Section 2) that there is at most one solution $(x,y,z)$ associated with a given ideal factorization in $C_D$ (with one easily handled exception).  Since $g = 2^{\omega(c)-1}$ and there are four parity classes as defined above, we obtain the result in the first sentence of Theorem A.  The result in the second sentence of Theorem A follows from a result in Hua \cite{H}.  

In this paper we will show that, for a given $(a,b,c)$, any solution $(x,y,z)$ must occur in one of at most two parity classes of $x$ and $y$.  We then show that any solution in a given parity class must be associated with one of at most two ideal factorizations in $C_D$, regardless of the number of distinct primes dividing $c$.  We also show that if solutions occur in two parity classes, then any solution in a given parity class must be associated with the same ideal factorization in $C_D$ as any other solution in the same parity class.  With these results we obtain: 

\begin{theorem}  
	For relatively prime integers $a$ and $b$ both greater than one and odd integer $c$, there are at most two solutions in positive integers $(x,y,z)$ to (1.1); any solution $(x,y,z)$ must satisfy $z < ab/2$.
\nopagebreak

	There are an infinite number of such $(a,b,c)$ giving exactly two solutions. 
\end{theorem} 

(The bound on $z$ follows from Theorem 3 of \cite{ScSt}.  The infinite family $(a,b,c: x_1,y_1, z_1; x_2,y_2,z_2) = (2,2^m - 1, 2^m+1: 1,1,1; m+2, 2,2)$ suffices to verify that there are an infinite number of $(a,b,c)$ giving exactly two solutions.)   

The result that there are at most two solutions to (1.1) in the special case in which either $x$ or $y$ is a fixed constant (usually 1) has been obtained by Bennett \cite{Be} using lower bounds on linear forms in logarithms when $\gcd(a,b)=1$ and using elementary methods when $\gcd(a,b)>1$.  Theorem 1 above provides an elementary proof of Bennett's result for the case $c$ odd.

\section{ Proof of Theorem 1}  

Throughout this proof we assume that $c$ is an odd integer and that $a$ and $b$ are relatively prime integers greater than 1.  We also assume throughout this proof that (1.1) has solutions $(x_1,y_1,z_1)$, $(x_2, y_2, z_2)$, \dots, $(x_n, y_n, z_n)$ such that there is no integer greater than 1 dividing all of $x_1$, $x_2$, \dots, $x_n$, and no integer greater than 1 dividing all of $y_1$, $y_2$, \dots, $y_n$, where $n>1$.  Note that this assumption (which may involve redefining $a$ and $b$) does not affect the number of solutions $(x,y,z)$, or the value of $D$ corresponding to a given solution, or the value of any $\gamma(x,y,z)$.  We will prove Theorem 1 by using four lemmas.    

\begin{lemma} 
For a given $(a,b,c)$, all solutions $(x,y,z)$ to (1.1) occur in at most two parity classes of $x$ and $y$.  
\end{lemma}

\begin{proof}
It suffices to show that, for any prime $p$ dividing $c$, the congruence 
$$ a^x \equiv - b^y \bmod p$$
has solutions $(x,y)$ in at most two parity classes.  Let $d$ be a primitive root of $p$ with $a \equiv d^r \bmod p$ and $b \equiv d^s \bmod p$, $ 0 \le r, s < p-1$.  So we are considering the congruence 
$$ d^{rx} \equiv - d^{sy} \equiv d^{ \frac{p-1}{2} + sy}  \bmod p $$
or, equivalently, 
$$ rx \equiv \frac{p-1}{2} + sy \bmod p-1. $$
Let $2^u || r$, $2^v || s$, and $2^w || \frac{p-1}{2}$, taking $u \le v$.  Then $u=v<w$ requires $ 2 \mid x-y$, $u=v=w$ requires $2 \notmid x-y$, and $u =v > w$ cannot occur, so that, when $u=v$, at most two parity classes of $x$ and $y$ are possible.  $u < \min(v,w)$ requires $ 2 \mid x$; $u=w<v$ requires $2 \notmid x$; and $w<u<v$ cannot occur; thus, in all cases, at most two parity classes of $x$ and $y$ are possible.  
\end{proof}

\begin{lemma} 
In a given parity class of $x$ and $y$, there is at most one solution $(x,y,z)$ to (1.1) associated with a given ideal factorization in $C_D$, except when $(a,b,c)=(3, 10, 13)$ or $(10,3,13)$.  
\end{lemma}

\begin{proof}
See the first paragraph of the proof of Theorem 2 of \cite{Sc}, noting that the only relevant instance of exception (iii) in Theorem 1 of \cite{Sc} is given by $(a,b,c)=(3,10,13)$.  
\end{proof}

We define $u(m)$ to be the least positive integer $t$ such that $m^t \equiv 1 \bmod c$.

\begin{lemma} 
If (1.1) has solutions in more than one parity class, then we must have either 
$a^{u(a)/2} \equiv -1 \bmod c$ or $b^{u(b)/2} \equiv -1 \bmod c$. 
\end{lemma}

\begin{proof}
Assume (1.1) has two solutions $(x_1,y_1,z_1)$ and $(x_2,y_2,z_2)$ such that $2 \notdiv x_1$ and $2 \mid x_2$.  We have
$ a^{x_1} \congruent -b^{y_1} \bmod c$ and $a^{x_2} \congruent -b^{y_2} \bmod c$ so that $a^{\vert x_1-x_2\mid} \congruent  b^{t} \bmod c$ for some $t$ 
such that $0 \le t < u(b)$.  Let $L = \lcm(x_1, x_1-x_2)$ so that $L$ is odd.  Then we have
$$ a^{L} \congruent (-b^{y_1})^{L/{x_1}} \congruent (b^t)^{L/{\vert x_1-x_2\mid}}  \bmod c,  $$
so that the congruence $b^q \congruent -1 \bmod c$ has a solution $q$.  This requires $b^{u(b)/2} \congruent -1 \bmod c$.  
	
Similarly, if (1.1) has two solutions $(x_1, y_1, z_1)$ and $(x_2,y_2, z_2)$   such that $2  \notdiv  y_1-y_2$, we must have $a^{u(a)/2} \congruent  -1 \bmod c$.    
\end{proof}

For the proof of Lemma 4 which follows, we use the following definition:  
we say that $h_1+k_1 \sqrt{-D} \congruent h_2+k_2 \sqrt{-D} \bmod c$ if $h_1 \congruent h_2 \bmod c$ and $k_1 \congruent k_2 \bmod c$.

\begin{lemma} 
All solutions to (1.1) in a given parity class are associated with one of at most two ideal factorizations in the set $C_D$.  If (1.1) has solutions in more than one parity class, then any two solutions in the same parity class must be associated with the same ideal factorization in the set $C_D$.   
\end{lemma}

\begin{proof}
	Let $(x_1,y_1,z_1)$ and $(x_2,y_2,z_2)$ be two solutions in the same parity class, with $x_1$ the least $x$ occurring in any solution in the parity class and $(x_2,y_2,z_2)$ any other solution in the parity class.  We have 
	$$ \gamma(x_1,y_1,z_1) = a^{x_1} - b^{y_1} + 2 \sqrt{-a^{x_1} b^{y_1}}  \eqno{(2.1)}$$
	and 
	$$ \gamma(x_2,y_2,z_2) \congruent  a^{x_2} - b^{y_2} + 2 a^{(x_2-x_1)/2} b^{d}\sqrt{-a^{x_1} b^{y_1}}  \bmod  c   \eqno{(2.2)} $$
	where $d = h u(b) + (y_2-y_1)/2$, where $h$ is any integer such that $d \ge 0$ (we can take $h$ to be the least such integer greater than or equal to zero so that $h=0$ when $y_2 \ge y_1$).  
	We have
	$$ a^{x_1} a^{x_2-x_1} = a^{x_2} \congruent -b^{y_2} \congruent -b^{y_1} b^{2d}  \bmod c.  \eqno{(2.3)}$$
	Since $a^{x_1} \congruent -b^{y_1} \bmod c$, (2.3) gives 
	$$ a^{x_2-x_1} \congruent b^{2d} \bmod c.  \eqno{(2.4)}  $$
	Since $b$ is prime to $c$, there exists an integer $\delta$ such that  
	$$ a^{(x_2-x_1)/2}  \congruent \delta b^{d} \bmod c, -(c-1)/2 \le \delta \le (c-1)/2.  \eqno{(2.5)} $$
	(2.5) with (2.4) gives 
	$$\delta^2 \congruent 1 \bmod c.  \eqno{(2.6)} $$
Now consider all the solutions $(x_i, y_i, z_i)$ to (1.1), regardless of parity class.  $a^{x_i} \congruent -b^{y_i} \bmod c$ for every $i$, and there is no integer greater than 1 dividing all the $x_i$ and also no integer greater than one dividing all the $y_i$.   So we can construct a linear combination of all the $x_i$, and also a linear combination of all the $y_i$, to obtain
	$$a \congruent \pm b^{t_1} \bmod c, b \equiv \pm a^{t_2} \bmod c  $$
	where $0 \le t_1 < u(b)$, $0 \le t_2 < u(a)$.  So from (2.5) we get
	$$ \delta \equiv \pm g^{r} \bmod c, -(c-1)/2 \le \delta \le (c-1)/2,  \eqno{(2.7)} $$
	where we fix either $g = a$ or $g=b$ and take $0 \le r < u(g)$.  By (2.6), $g^{2r} \congruent 1 \bmod c$.  
	This requires either $r=0$ or $r=u(g)/2$.  Let $A$ be the set of at most four residues modulo $c$ which could be congruent to $\pm g^r \bmod c$ when $g=a$, and let $B$ be the set of at most four residues modulo $c$ which could be congruent to $\pm g^r \bmod c$ when $g=b$. $\delta \in A$ and $\delta \in B$.  Notice that $A$ and $B$ are independent of the choice of $(x_2,y_2, z_2)$, and independent of the choice of parity class.  

Recalling (2.4) and (2.5) we see that (2.2) becomes 
	$$ \gamma(x_2, y_2, z_2) \congruent b^{2d} (a^{x_1} - b^{y_1}) + 2 b^{2d} \delta \sqrt{-a^{x_1} b^{y_1} }  \bmod c. \eqno{(2.8)} $$
Write $\gamma_1 = \gamma(x_1,y_1,z_1)$, $\gamma_2= \gamma(x_2, y_2, z_2)$.   Let $\beta = a^{x_1} - b^{y_1} + 2 \delta \sqrt{- a^{x_1} b^{y_1}}$.  Since $\gcd(b,c)=1$, (2.8) gives $c \mid \beta \conj{\beta}$.  Let $\German{c}_k$ be the unique ideal such that $\German{c}_k \conj{\German{c}_k} \in C_D$ and 
$$\German{c}_k \mid [\beta] =  [a^{x_1} - b^{y_1} + 2 \delta \sqrt{- a^{x_1} b^{y_1}}]. \eqno{(2.9)}$$
   $\German{c}_k$ contains both $\beta$ and $c$, so by (2.8) it contains $\gamma_2$, so that $\German{c}_k \mid [\gamma_2]$, so that the solution $(x_2,y_2,z_2)$ is associated with the ideal factorization $\German{c}_k \conj{\German{c}_k}$.   Since either $r=0$ or $r=u(g)/2$ ($r=u(g)/2$ is possible only when $2 \mid u(g)$), there are at most two possible values for $g^r \bmod c$ one of which must be $1 \bmod c$, so, by (2.7) and (2.9), there are at most two possible choices for the pair $\{ \German{c}_k, \conj{\German{c}_k} \}$ one of which must be $\{ \German{c}_h, \conj{\German{c}_h} \}$ where $\German{c}_h \mid [\gamma_1]$ or $[ \conj{\gamma_1}]$, so that $\German{c}_h \conj{\German{c}_h} $ is the ideal factorization with which the solution $(x_1,y_1,z_1)$ is associated.  If (1.1) has solutions in more than one parity class, then, by Lemma 3, we can assume $g^{u(g)/2} \congruent -1 \bmod c$ so that, by (2.7), $\delta = \pm 1$ and either $\beta = \gamma_1$ or $\beta = \conj{\gamma_1}$, so that $\German{c}_k = \German{c}_h$ or $\German{c}_k = \conj{\German{c}_h}$, so that the solution $(x_2,y_2,z_2)$ is associated with the same ideal factorization as the solution $(x_1,y_1,z_1)$.  
\end{proof}

	Recalling the parenthetical comments immediately following the statement of Theorem 1, we see that Theorem 1 follows from Lemmas 1, 2, and 4, except when $(a,b,c) = (3,10,13)$ or $(10,3,13)$.  If $(a,b,c) = (3,10,13)$, then consideration modulo 13 gives $2 \nmid y$; since $3^2 + 2^5\cdot 5 = 13^2$, there can be no solution with $2 \mid x$, by Theorem 1 of [9]; thus there are exactly two solutions by Theorem 1 of [9].

\section{Cases with exactly two solutions}  

We give the following conjecture, which allows $c$ even as well as $c$ odd:

\begin{Conjecture}    
For integers $a$, $b$, and $c$ all  greater than one with  $\gcd(a,b)  = 1$, there is at most one solution in positive integers $(x,y,z)$ to (1.1) except for the following $(a,b,c)$ or $(b,a,c)$:  $(5, 2, 3)$, $(7, 2, 3)$,  $(3, 2, 11)$,  $ (3, 2, 35)$, $ (3, 2, 259)$,  $(3, 4, 259)$,  $(3, 16, 259)$,  $(5, 2, 133)$,  $(3, 10, 13)$,  $(89, 2, 91)$,  $(91, 2, 8283)$, $(3, 5, 2)$, \allowbreak $(3,13,2)$,  $(3, 13, 4)$,  $(3, 13, 16)$,  $(3, 13, 2200)$, and   $(2^{n} - 1, 2, 2^{n} + 1)$ for any positive integer $n \ge 2$. 
\end{Conjecture} 

If $c \ne 2200$ then Theorem 1 of this paper with Theorem 6 of [9] can be used to show that all $(a,b,c)$ in the above list give exactly two solutions in positive integers $x$, $y$, and $z$ except for $(a,b,c) = (3, 5, 2)$, which has three solutions.  If $c=2200$ and $\{ a,b \} = \{ 3, 13 \}$, then, if $z>1$, consideration modulo 16 gives $x \equiv y \bmod 4$ in (1.1), contradicting consideration modulo 5; so $z=1$, giving exactly two solutions.  

A computer search found no other $(a,b,c)$ with $\gcd(a,b)=1$ giving more than one solution in positive integers $x$, $y$, $z$ to (1.1) in the ranges $a<2500$, $b<10000$ with $a^x < 10^{30}$, $b^y < 10^{30}$.  

Except for $(5,2,3)$, $(7,2,3)$, and $(2^n-1,2,2^n+1)$, all the $(a,b,c)$ in the list given in Conjecture 1 can be derived from the first six entries of Conjecture 1.2 of \cite{Be} which deals with the case in which one of $x$ or $y$ is constant.

\end{document}